\documentclass[a4paper,11 pt]{amsart}
\usepackage{hyperref}

\usepackage[usenames,dvipsnames,svgnames,table]{xcolor}
\usepackage[all]{xy}
 \usepackage{enumitem}
 \usepackage[normalem]{ulem}

\newcommand{\tsk}[1]{\textcolor{YellowOrange}}

\usepackage{tikz}
\usepackage{tikz-cd}

\makeatletter
\def\@endtheorem{\endtrivlist}% NEW
\makeatother
\usepackage[active]{srcltx}
\usepackage{amsmath}
\usepackage{amssymb}
\usepackage{amscd}
\usepackage{amsthm}
\usepackage[latin1]{inputenc}
\usepackage{mathrsfs}  

\usepackage{nicefrac}
\renewcommand{\Re}{\operatorname{Re}}

\newcommand{\numero}{9}

\newtheorem{teo}{Theorem}[section]
\newtheorem{defin}[teo]{Definition}
\newtheorem{prop}[teo]{Proposition}
\newtheorem{cor}[teo]{Corollary}
\newtheorem{lemma}[teo]{Lemma}
\theoremstyle{definition}

\newtheorem{remark}[teo]{Remark}

\newtheoremstyle{dico}% name of the style to be used
 {\baselineskip}   % ABOVESPACE
  {\topsep}   % BELOWSPACE
  {}  % BODYFONT
  {0pt}       % INDENT (empty value is the same as 0pt)
  {} % HEADFONT
  {.}         % HEADPUNCT
  {5pt plus 1pt minus 1pt} % HEADSPACE
  {}          % CUSTOM-HEAD-SPEC
\theoremstyle{dico}
\newtheorem{say}[teo]{}
\numberwithin{equation}{section}

\newcommand{\ra}{\rightarrow}
\newcommand{\C}{\mathbb{C}}
\newcommand{\R}{\mathbb{R}}
\newcommand{\Zeta}{{\mathbb{Z}}}

\newcommand{\QQ}{{\mathbb{Q}}}
\newcommand{\meno}{^{-1}}

\newcommand{\alfa}{\alpha}

\newcommand{\vacuo}{\emptyset}

\newcommand{\La}{\Lambda}

\newcommand{\restr}[1]          {\vert_{#1}}
\newcommand{\Aut}{\operatorname{Aut}}

\newcommand{\End}{\operatorname{End}}

\newcommand{\spur}{\operatorname{Tr}}
\renewcommand{\setminus}{-}

\newcommand{\eps}{\varepsilon}
\renewcommand{\phi}{\varphi}
\newcommand{\lds}{\ldots}
\newcommand{\cds}{\cdots}
\newcommand{\cd}{\cdot}

\newcommand{\sx}{\langle}
\newcommand{\xs}{\rangle}

\newcommand{\ga}{\gamma}
\newcommand{\Ga}{\Gamma}

\newcommand{\gr}{\mathsf{g}}
\newcommand{\Fix}{\mathsf{Fix}}

\newcommand{\id}{\operatorname{id}}

\newcommand{\GL}{\operatorname{GL}}

\newcommand{\PP}{\mathbb{P}}

\renewcommand{\phi}             {\varphi}

\newcommand{\HH}{\mathfrak{H}}
\newcommand{\sieg}{\HH_g}
\newcommand{\mt}{\operatorname{MT}}

 \newcommand{\Sl}                {\operatorname {SL}}
 \newcommand{\psl}                {\operatorname {PSL}}

\newcommand{\Sp}                {\operatorname {Sp}}
     
 \newcommand{\tr}              {\operatorname {tr}}

\newcommand{\mm}{{\mathbf{m}}}
\newcommand{\jac}{\mathsf{T}^0_g}
\newcommand{\tor}{\mathsf{T}_g}

\newcommand{\ag}{\mathsf{A}_g}
\newcommand{\mg}{\mathsf{M}_g}
\newcommand{\zg}{\mathsf{Z}}
\newcommand{\datum}{{(\mm, G, \theta)}}
\newcommand{\zgm}{\mathsf{Z}(\mm, G, \theta)}
\newcommand{\Diff}{\operatorname{Diff}}
\newcommand{\Map}{\operatorname{Map}}
\newcommand{\braid}{\mathbf{B_r}}
\newcommand{\ut}{U_t}

\begin{document}

\author{Paola Frediani, Alessandro Ghigi and Matteo Penegini}

\title{Shimura varieties in the Torelli locus via Galois coverings}

\address{Universit\`{a} di Pavia} \email{paola.frediani@unipv.it}
\address{Universit\`a di Milano Bicocca}
\email{alessandro.ghigi@unimib.it} \address{Universit\`{a} di Milano}
\email{matteo.penegini@unimi.it}

\thanks{ The first and second authors were partially supported by PRIN
  2012 MIUR ''Moduli, strutture geo\-me\-tri\-che e loro
  applicazioni''.  The first author was partially supported also by
  FIRB 2012 ''Moduli spaces and applications'' and by a grant of
  Max-Planck Institut f\"ur Mathematik, Bonn.  The second author was
  supported also by FIRB 2012 ''Geometria differenziale e teoria
  geometrica delle funzioni''.  The third author was partially
  supported by PRIN 2010 MIUR ``Geometria delle Variet\`a Algebriche".
} 

\begin{abstract}
  Given a family of Galois coverings of the projective line, we give a
  simple sufficient condition ensuring that the closure of the image
  of the family via the period mapping is a special (or Shimura)
  subvariety of $A_g$. By a computer program we get the list of all
  families in genus $g \leq \numero$ satisfying our condition.  There
  are no families with $g=8,9$, all of them are in genus $g \leq 7$.
  These examples are related to a conjecture of Oort.  Among them we
  get the cyclic examples constructed by various authors (Shimura,
  Mostow, De Jong-Noot, Rohde, Moonen and others) and the abelian
  non-cyclic examples found by Moonen-Oort.  We get 7 new non-abelian
  examples.
\end{abstract}

\maketitle

\tableofcontents{}

\section{Introduction}

\begin{say}
  Denote by $\ag$ the moduli space of principally polarized abelian
  varieties of dimension $g$ over $\C$, by $\mg$ the moduli space of
  smooth complex algebraic curves of genus $g$ and by $j \colon \mg
  \ra \ag$ the period mapping or Torelli mapping.  We set $\jac:=j
  (\mg)$ and call it the open Torelli locus.  The closure of $\jac$ in
  $\ag$ is called the \emph{Torelli locus} (see e.g.
  \cite{moonen-oort}) and is denoted by $\tor$.  From the complex
  analytic point of view, $\ag= \Sp(2g, \Zeta) \backslash
  \mathfrak{H}_g $, where $\mathfrak{H}_g$ is the Siegel upper
  half-space.  Therefore $\ag$ has a natural structure of complex
  analytic orbifold and the symmetric metric on $\mathfrak{H}_g$
  descends to a locally symmetric orbifold metric on $\ag$.  We will
  always consider this metric on $\ag$.  It is an interesting problem
  to study the metric properties of the inclusion $\jac \subset
  \ag$. The moduli space of curves also admits a natural structure of
  complex orbifold and the period mapping is an orbifold map. Moreover
  outside the hyperelliptic locus the period mapping is an orbifold
  immersion \cite{oort-steenbrink}.  This allows to study $\jac$
  (outside the hyperelliptic locus) using Riemannian geometry,
  i.e. via the second fundamental form. This is the direction taken in
  \cite{cpt}, \cite{cf1}, \cite{cf2}, \cite{cfg}.  One expects that
  $\jac$ be very curved inside $\ag$. For example the second
  fundamental form should be in some sense non-degenerate and in
  particular $\jac$ should contain very few totally geodesic
  submanifolds of $\ag$. Among the results in this direction we
  mention the following ones. Let $\mathsf{Z} $ be a totally geodesic
  subvariety of $ \ag$ such that $\zg\subset \tor$ and $\zg\cap \jac
  \neq \vacuo$. Toledo \cite{toledo} considered the case when $\zg$ is
  a compact curve and obtained an upper bound for the area and some
  curvature restrictions for $\zg$. Hain \cite{hain} and later de Jong
  and Zhang \cite{dejong-zhang} proved under some conditions, that if
  $\zg$ is a locally symmetric variety uniformized by an irreducible
  symmetric domain, this must be the complex ball.  (Recall that a
  submanifold of $\ag$ is totally geodesic if and only if it is a
  locally symmetric submanifold.)
  % This follows from the fact that the local geodesic symmetry of
  % $\ag$ preserves a totally geodesic submanifold using \cite{sakai},
  % Prop. 6.3 (3)).
  Very recently Liu, Sun, Yang and Yau \cite{liu-yau-ecc} got the same
  result by differential geometric techniques, under the assumption
  that $\zg$ is contained in $\jac$.  In \cite{cfg} Colombo and the
  first two authors used the second fundamental form to get an upper
  bound for the dimension of $\zg$ depending only on the genus.  Other
  related papers include \cite{lu-zuo-Mumford-prep},
  \cite{grushevsky-moeller-prep}.
\end{say}

\begin{say}
  The stack $\ag$ (or equivalently its associated complex analytic
  orbifold) parametrizes Hodge structures of weight 1 on a lattice of
  rank $2g$. On $\ag$ there is a natural variation of Hodge structure
  over $\mathbb{Q}$ (in the orbifold sense), whose fibre over $A$ is
  $H^1(A, \mathbb{Q})$.  The Hodge loci for this variation of Hodge
  structure are called \emph{special subvarieties} or \emph{Shimura
    subvarieties}, see \cite[\S 3.3]{moonen-oort}.  The special
  varieties are totally geodesic and an important theorem of Moonen
  \cite{moonen-linearity-1} says that an algebraic totally geodesic
  subvariety of $\ag$ is special if and only if it contains a CM
  point.  Arithmetical consideration led Oort \cite{oort-can} to the
  following expectation: for large $g$ there should be no
  positive-dimensional special subvariety $\zg$ of $\ag$, such that
  $\zg \subset \tor$ and $\zg\cap\jac\neq \vacuo$.  See \cite[\S
  4]{moonen-oort} for more details.  On the other hand, for low genus
  there are examples of such $\zg$ (see
  \cite{shimura-purely-transcendental, mostow-discontinuous,
    dejong-noot, rohde, moonen-special} and also the survey \cite[\S
  5]{moonen-oort}.) All the examples known so far are in genus $\leq
  7$ and are constructed using abelian Galois covers of the line.
\end{say}

\begin{say}
  The purpose of this paper is, first of all, to give a simple
  sufficient condition for a family of Galois covers of the line to
  yield a Shimura variety (see Theorem \ref{mainteoA} below). This
  criterion simplifies and extends the previous arguments.  Next, we
  apply it to construct new examples of such families for non-abelian
  Galois coverings. Moreover we analyze in detail the geometry of all
  examples, both with abelian and with non-abelian Galois group,
  giving the complete list of all the distinct Shimura families in
  genus $g \leq \numero $ obtained using this criterion.

  A Galois covering of $C \ra \PP^1$ is determined by the ramification
  data $\mm:=(m_1, \lds, m_r)$, the Galois group $G$, an epimorphism
  $\theta : \Ga_r \ra G$ and the branching points $t_1, \lds, t_r
  \in \PP^1$, (see \S \ref{covering-section} for the notation).
  Fixing the datum $(\mm, G, \theta)$ and letting the points $t_j$
  vary, one gets a family of curves and a corresponding family of
  Jacobians.  Denote by $\zg(\mm, G, \theta) $ the closure of this set
  of Jacobians in $\ag$. It is an $(r-3)$--dimensional subvariety of
  $\ag$.  If $C \ra \PP^1$ is one of the coverings, consider the
  representation $\rho$ of $G$ on $H^0(C,K_C)$ and on its symmetric
  power $S^2 H^0(C, K_C)$. We set
  \begin{gather*}
    N:= \dim (S^2 H^0(C, K_C)) ^ G.
  \end{gather*}
  Both the isomorphism class of $\rho$ and the number $N$ depend only
  on the datum $\datum$, not on the particular element $C$ of the
  family.
\end{say}

\begin{teo} [see Theorem \ref{criterio}]
  \label{mainteoA}
Let $(\mm, G, \theta)$ be a datum as above.  Assume that
  \begin{gather}
    \label{bona}
    \tag{$\ast$} N = r-3.
  \end{gather}
  Then $\zg (\mm, G, \theta)$ is a special subvariety of PEL type of
  $\ag$, such that $\zg (\mm, G, $ $ \theta)$ $ \subset \tor$ and
  $\zg(\mm, G, \theta)\cap \jac \neq \vacuo$.
\end{teo}

Observe that when $r=3$ and $N=0$ this yields a criterion for a
Jacobian to have complex multiplication, see Corollaries \ref{ACM} and
\ref{JCM}.

\begin{say}  
  The condition in Theorem \ref{mainteoA} already appears in
  \cite[Prop. 5.4]{cfg}. There it is shown that under this condition
  $\zg(\mm,G,\theta)$ is totally geodesic.  The proof uses the second
  fundamental form of the family of Jacobians. Since special
  subvarieties are totally geodesic, the theorem above strenghtens the
  result in \cite {cfg} with a different proof.

  We have used the criterion in Theorem \ref{criterio} for a
  systematic search of special subvarieties of the form $\zgm$. At the
  beginning, especially in genus 4, we used a classification of the
  groups acting on algebraic curves from the point of view of the
  representation on holomorphic 1-forms. This classification is
  available in genus $g \leq 5$, thanks to the efforts of Akikazu
  Kuribayashi, Izumi Kuribayashi and Hideyuki Kimura
  \cite{kuribayashi+akikazu-families,kuribayashi-linear,
    kurikuri,kuribayashi-akikazu-kimura}. See also
  \cite{magaard-e-soci}.  Breuer \cite{breuer} has made a systematic
  computation of the possible automorphism groups for all the curves
  of genus $g\leq 48$.  For the calculations done in this paper we
  used the computer algebra program \verb|MAGMA| \cite{MA}. Our script
  is available at:

  \verb|users.mat.unimi.it/users/penegini/|

  \verb|publications/PossGruppigFix_v2Hwr.m|

  \smallskip

\noindent
Using this script we determine all the families $\zg \datum$ with
genus $g \leq \numero$ and we compute the number $N$, checking which
families satisfy the condition of Theorem \ref{criterio}.  Our results
are summarized in the following.
\end{say}

\begin{teo}
  \label{mainteoB} For genus $g\leq \numero$ there are exactly 40 data
  $(\mm, G, \theta)$ such that $N=r-3 > 0$.  For these 40 data the
  image $\zgm$ is a special subvariety of $\ag$ of positive dimension,
  which is contained in $\tor$ and intersects $\tor^0$.  Among these
  data there are 20 cyclic ones and 7 abelian non-cyclic ones.  The
  remaining 13 have non-abelian Galois group.  All these data occur in
  genus $g\leq 7$.
\end{teo}

The 20 cyclic data have been found in
\cite{shimura-purely-transcendental,
  mostow-discontinuous,dejong-noot,rohde, moonen-special} and the 7
abelian non-cyclic data have been found in \cite[\S 5]{moonen-oort}.
The 13 non-abelian data are new.  Professor Xin Lu informed us that
one of the non-abelian families has already been studied very recently
and from a different point of view in
\cite[Ex. 7.2]{lu-zuo-Mumford-prep}.  See Table \ref{data} for the
list of all the 40 data.

  \begin{say}
    It should be remarked that as far as we know the condition
    \eqref{bona} is only sufficient, but not necessary for $\zgm$ to
    be special.  So one cannot exclude that some datum $(\mm, G,
    \theta)$ with $N > r-3$ gives a special $\zgm$. For the case of
    cyclic coverings this has been ruled out by Moonen
    \cite{moonen-special} using deep results in arithmetic
    geometry. Thus in the case of cyclic groups the condition that
    $N=r-3$ is both sufficient and necessary.
  \end{say}

\begin{say}
  It can happen that two different data $\datum$ and $(\mm', G',
  \theta')$ give rise to the same subvariety in $\ag$,
    i.e. $\zg\datum = \zg (\mm', G', \theta') $.  In fact the 40 data
  we found do not give rise to 40 different subvarieties. In
  \S \ref{id-section} we describe systematically this phenomenon and
  we get the complete list of the distinct subvarieties, which
  is summarized in the following theorem.
\end{say}

\begin{teo}
  \label{mainC}
  The 40 data satisfying \eqref{bona} yield exactly 30 distinct
  Shimura subvarieties, which are listed in Table
  \ref{27}. The numbers refer to the data listed in Table \ref{data}.
\end{teo}

\begin{table}[h]
  \label{27}
  \caption{Distinct Shimura subvarieties}
  \begin{center}
    \begin{tabular}[t] {|c|c|m{.35\linewidth}|}
      \hline
      $g$ & $\dim $  & families \\
      \hline \hline
      $1$ & $1$ & $(1)= (21)$\\
      \hline \hline
      2 &  $1$ & 
      $(3) =(5) = (28) = (30) $
      $(4) = (29) $ \\
      \hline
      2      & $2$ & $(26)$ \\
      \hline
      2     & $3$ & $(2)$ \\
      \hline \hline
      $3$ & $1$ & 
      \begin{minipage}[t]{1\linewidth}
        $  (7) = (23) = (34)$ \\
        $ (9)$  \\
        $  (22) $ \\
        $ (33) = (35)$
      \end{minipage}
      \\
      \hline
      3     &      $2$ &
      \begin{minipage}[t]{1\linewidth}
        $  (6)$  \\
        $(8)$  \\
        $(31)$   \\
        $(32)$
      \end{minipage}
      \\
      \hline 
      3    & $3 $ &$ (27)$ \\
      \hline
    \end{tabular}
    \hspace{3EM}
    \begin{tabular}[t]{|c|c|m{.2\linewidth}|}
      \hline
      $g$ & $\dim $  & families\\
      \hline 
      \hline
      $4$ &      $   1$ & 
      \begin{minipage}[t]{1\linewidth}
        $  (11)$  \\
        $  (12)$ \\
        $  (13) = (24)$ \\
        $  (25) = (38) $ \\
        $  (36)$ \\
        $ (37) $
      \end{minipage}
      \\
      \hline
      4      &       $2$ & $   (14)$\\
      \hline
      4     &      $3 $ &$ (10)$\\
      \hline \hline
      $5$ & $1$ & 
      \begin{minipage}[t]{1\linewidth}
        $(15)$ \\
        $(39)$
      \end{minipage}
      \\
      \hline 
      \hline
      $6$ &      $   2$ & 
        $ (16) $
            \\
      \hline
      6      &       $1$ & $   (17)$\\
      \hline
      6     &      $1 $ &$ (18)$\\
      \hline
\hline
      $7$ & $1$ &
      \begin{minipage}[t]{1\linewidth}
        $  (19)$ \\
        $(20)$\\
        $ (40)$
      \end{minipage}
      \\
      \hline
    \end{tabular}
  \end{center}
\end{table}

In particular in genus 2 all families are abelian. There are three
non-abelian Shimura families in genus 3, two in genus 4, one in genus
5 and one in genus 7.

One of the two non--abelian families in genus 4 (family (36)) and the
non--abelian family in genus 5 (family (39)) are contained in the
hyperelliptic locus, see \ref{HE-4}. All the other non--abelian
Shimura families are not contained in the hyperellitic locus, see
\ref{HE-1}, \ref{HE-2} and the proof of Theorem \ref{id-3}.

\begin{say}
  The plan of the paper is the following.

  In \S \ref{covering-section} we fix the notation and we recall some
  preliminary results on families of Galois coverings of the
  projective line.

  In \S \ref{Shimura-section} we give a brief summary of definitions
  and results on special subavarieties of $\ag$, especially those of
  PEL type. Next we prove Theorem \ref{mainteoA}.

  Section \S \ref{examples-section} is devoted to the new
  examples.  In two sample cases we do the computation of $N$ by hand.
  We also make some additional remarks on the hyperellipticity of the
  families and on inclusions between them.

  Section \S \ref{id-section} is devoted to the proof of Theorem
  \ref{mainC}.

  In the Appendix we explain the computations performed by the
  script. Table \ref{data} contains the list of all data satisfying
  \eqref{bona}.
%   For all the data with abelian Galois group the explicit
%   description of the epimorphism $\theta$ can be found in
%   \cite{moonen-oort}.  For the data with non--abelian Galois group
%   this information is given in \ref{list}.  

\end{say}

\medskip

{\bfseries \noindent{Acknowledgements.} }  We wish to thank Elisabetta
Colombo for many interesting conversations, which led us to attack
this problem. We also wish to thank Bert van Geemen for crucial help
with Proposition \ref{prop-Bert} and Ben Moonen for useful emails.  We
thank the anonymous referees for reading very carefully the manuscript
and for many questions and suggestions that helped us to improve the
paper a lot.  The first and second authors wish to thank the
Max-Planck Institut f\"ur Mathematik, Bonn for excellent conditions
provided during their visit at this institution, where part of this
work was prepared.

\section{Galois coverings of the line}
\label{covering-section}

\begin{say} \label{theo: Riemann} For any integer $r\geq 3$ let
  $\Ga_r$ denote the group with presentation $\Ga_r=\sx \ga_1, \lds,
  \ga_r | \ga_1\cds \ga_r =1\xs$.
  % Fix an $r$-tuple $\mm:=(m_1, \ldots ,m_r)$ of positive integers.
  % If $G$ is a finite group an epimorphism $\theta\colon \Gamma_r
  % \rightarrow G$ is called \emph{admissible of type} $\mathbf{m}$
  % if $\theta (\ga_i)$ has order $m_i$ for $i=1, \lds, r$.

\begin{defin}
  A \emph{datum} is a triple $\datum$, where $\mm :=(m_1, \ldots ,m_r)
  $ is an $r$-tuple of integers $m_i \geq 2$, $G$ is a finite group
  and $\theta : \Ga_r \ra G$ is an epimorphism such that $\theta
  (\ga_i)$ has order $m_i$ for each $i$.
\end{defin}
% Given a finite subset $t := \{t_1 , \ldots ,t_r \} \subset \PP^1$ of
% cardinality $r \geq 0$ set
Let $t: = (t_1, \lds, t_r)$ be an $r$-tuple of distinct points in
$\PP^1$.  Set $\ut := \PP^1\setminus \{t_1, \lds, t_r\}$ and choose a
base point $t_0 \in \ut$. By elementary topology there exists an
isomorphism $\pi_1(\ut, t_0 ) \cong \Ga_r$ such that the element
$\gamma_i$ corresponds to a simple closed loop winding around the
point $t_i$ counterclockwise. If $f\colon C \longrightarrow \PP^1$ is
a Galois cover with branch locus $t$, set $V := f^{-1}(\ut)$. Then
$f\restr{V} : V \ra \ut$ is an unramified Galois covering. Let $G$
denote the group of deck transformations of $f\restr{V}$.  Then there
is a surjective homomorphism $%\theta^{\prime}\colon
\pi_1(\ut, t_0 ) \longrightarrow G$, which is well-defined up to
composition by an inner automorphism of $G$.  Since $\Ga_r \cong
\pi_1(\ut,t_0)$ we get an epimorphism $\theta : \Ga_r \ra G$. If $m_i$
is the local monodromy around $t_i$ and $\mm=(m_1, \lds, m_r)$, then
$(\mm, G, \theta)$ is a datum.  Thus a Galois cover of $\PP^1$
branched over $t$ gives rise -- up to some choices -- to a datum. The
Riemann's existence theorem ensures that the process can be reversed:
a branch locus $t$ and a datum determine a covering of $\PP^1$ up to
isomorphism (see e.g. \cite[Sec. III, Corollary 4.10]{M95}).  We wish
to show that the process can be reversed also in families, namely that
to any datum is associated a family of Galois covers of $\PP^1$.
\end{say}

\begin{say}
  \label{family}
  In fact let $\datum$ be a datum.  Set $Y_r:=\{t=(t_1, \lds, t_r) \in
  (\PP^1)^r : t_i\neq t_j $ for $i\neq j\}$.  Fix a point $t \in Y_r$,
  a base point $t_0 \in \ut$ and an isomorphism $\Ga_r \cong
  \pi_1(\ut, t_0)$ (this is equivalent to choosing a point in the
  Teichm\"uller space $T_{0,r}$). By the above we get a $G$-cover $C_t
  \ra\PP^1$ branched at the points $t_i$ with local monodromies
  $m_1,\dots,m_r$.  This yields a monomorphism of $G$ into the mapping
  class group $ \Map_g := \pi_0 ( \Diff^+ (C_t))$.  Denote by $T_g^G$
  the fixed point locus of $G$ on the Teichm\"uller space $T_g$. It is
  a complex submanifold of dimension $r-3$, isomorphic to the
  Teichm\"uller space $T_{0,r}$ (see e.g. \cite{baffo-linceo,gavino}).
  This isomorphism can be described as follows: if $(C,\phi) $ is a
  curve with a marking such that $[(C, \phi)] \in T_g^G$, the
  corresponding point in $T_{0,r}$ is $[(C/G, \psi, b_1, \lds, b_r)]$,
  where $\psi$ is the induced marking (see \cite{gavino}) and $b_1,
  \lds, b_r$ are the critical values of the projection $C \ra C/G$.
  % This space $T_gG$ parametrizes points in $T_g$ corresponding to
  % marked curves $C$ with $G\subset \Aut(C)$ whose topological action
  % is fixed by the

  We remark that on $T_g^G$ we have a universal family $\mathcal{C}
  \to T_g^G$ of curves with a $G$--action. It is simply the
  restriction of the universal family on $T_g$.

  The diagonal action of $\psl(2, \C)$ on $(\PP^1)^r$ preserves
  $Y_r$. Thus $\psl(2, \C)$ acts on $Y_r$.  The map $Y_r \ra \mg$, $t
  \mapsto [C_t]$ is $\psl(2,\C)$--invariant, so we get a map $Y_r /
  \psl(2,\C) \ra \mg$.  Since $Y_r / \psl(2,\C) $ is isomorphic to
  $\mathsf{M}_{0,r}$, there is a surjective map $T_{0,r} \ra Y_r
  /\psl(2, \C)$.  Since $T_g^G \subset T_g$ there is also a map $T_g^G
  \ra \mg$ which has discrete fibres.  Recalling the description of
  the isomorphism $T_g^G \cong T_{0,r}$ one can easily check that the
  following diagram commutes:
  \begin{equation*}
    \begin{tikzcd}%[column sep=large]
      T_{0,r} \arrow{r}    & M_{0,r} \cong Y_r /\psl(2, \C) \arrow{r} & \mg.   \\
      & T_g^G \arrow{ul} {\cong} \arrow{ur} &
    \end{tikzcd}
  \end{equation*}
  We denote by $\mathsf{M}\datum$ the image of $Y_r$ in $\mg$, which
  is equal to the image of $T_g^G$. It is an irreducible algebraic
  subvariety of the same dimension as $T_g^G\cong T_{0,r}$, i.e. $r-3$
  (see e.g. \cite{baffo-linceo,gavino}).  Applying the Torelli map to
  $\mathsf{M}\datum$ one gets a subset of $\ag$.  We let
  $\zg(\mm,G,\theta)$ denote the closure of this subset in $\ag$.  By
  the above it is an algebraic subvariety of dimension $r-3$.

  Different data $\datum$ and $(\mm, G, \theta')$ may give rise to the
  same subvariety of $\mg$. This is related to the choice of the
  isomorphism $\Ga_r \cong \pi_1(U_t, t_0)$.  The change from one
  choice to another can be described using an action of the braid
  group ${\bf B_r}:=\langle \sigma_1, \ldots ,\sigma_{r}| \, \sigma_i
  \sigma_j = \sigma_j \sigma_i \, \, {\rm for } \, \, |i-j|\geq 2, \,
  \sigma_{i+1}\sigma_i\sigma_{i+1}=\sigma_i\sigma_{i+1}\sigma_i
  \rangle$.  There is a morphism $\phi : \braid \ra \Aut(\Ga_r)$
  defined as follows:
  \begin{gather*}
    \phi( \sigma_i) (\gamma_i) = \gamma_{i+1}, \quad \phi(\sigma_i)
    (\gamma_{i+1}) = \gamma_{i+1} ^{-1} \gamma_i \gamma_{i+1}, \\
    \phi(\sigma_i) (\gamma_j ) = \gamma_j \quad \text{for }j \neq i,
    i+1.
  \end{gather*}
  % \begin{gather*}
  %   \sigma_i (\gamma_i) = \gamma_{i+1}, \quad \sigma_i
  %   (\gamma_{i+1}) = \gamma_{i+1} ^{-1} \gamma_i \gamma_{i+1}, \quad
  %   \sigma_i (\gamma_j ) = \gamma_j \quad \text{for }j \neq i, i+1.
  % \end{gather*}
  Thus we get an action of $\braid$ on the set of data: $ \sigma \cd
  (\mm, G, \theta) : = (\sigma (\mm), G, \theta \circ
  \phi(\sigma\meno))$, where $\sigma (\mm) $ is the permutation of
  $\mm$ induced by $\sigma$.  Also the group $\Aut(G)$ acts on the set
  of data by $\alfa \cd (\mm, G, \theta) : = (\mm, G, \alfa \circ
  \theta )$.
  % \begin{itemize}
  % \item $\sigma_i (\gamma_i) = \gamma_{i+1}$
  % \item $\sigma_i (\gamma_{i+1}) = \gamma_{i+1} ^{-1}\gamma_i
  %   \gamma_{i+1}$
  % \item $\sigma_i (\gamma_j ) = \gamma_j $ for $ j \neq i, i+1$.
  % \end{itemize}
  The orbits of the $\braid \times \Aut(G)$--action are called
  \emph{Hurwitz equivalence classes}. Data in the same class give rise
  to the same subvariety $\mathsf{M}\datum$ and hence to the same
  subvariety $\zg\datum \subset \ag$. For more details see
  \cite{penegini2013surfaces,baffo-linceo,birman-braids}.
\end{say}

\begin{say}
  \label{say-ssg}
  Given a positive integer $m$ set $ \zeta_m = e^{2\pi i / m} $ and
  \begin{gather*}
    I(m) := \{\nu \in \Zeta: 1\leq \nu < m, \gcd(\nu, m) =1 \}.
  \end{gather*}
  If $\datum$ is a datum, set $x_i := \theta(\ga_i)$. The $r$-tuple
  $(x_1, \lds, x_r)$ is called a \emph{spherical system of
    generators}.  If $C$ is a curve with a $G$--action with datum
  $\datum$, then the cyclic subgroups $\left\langle
    x_{i}\right\rangle$ and their conjugates are the non-trivial
  stabilizers of the action of $G$ on $C$.  The action of the
  stabilizers near the fixed points can be completely described in
  terms of the epimorphism $\theta$, see \cite[Theorem 7]{Ha71}.  In
  particular we need the following results.  Suppose that an element
  $\gr \in G$ fixes a point $P \in C$.  Let $m$ be the order of $\gr$.
  The differential $d\gr_P$ acts on $T_P C$ by multiplication by an
  $m$-th root of unity $\zeta_P(\gr)$. The action can be linearized in
  a neighbourhood of $P$, i.e.
  % \cite[p. 97]{cartan-pro-lefschetz-suo}
  there is a local coordinate $z$ centered in $P$, such that $\gr $
  acts as $ z \mapsto \zeta_{P}(\gr ) z$.  Thus $\zeta_P(\gr )$ is a
  primitive $m$-th root of unity.  (See also \cite[Cor. III.3.5
  p. 79]{M95}.)  Denote by $\Fix(\gr)$ the set of fixed points of
  $\gr$. For $\nu \in I(m)$ set
  \begin{gather*}
    \Fix_\nu (\gr ) : = \{P \in C : {\bf g } P = P , \zeta_P (\gr ) =
    \zeta_m^\nu\}.
  \end{gather*}

\end{say}

\begin{lemma}
  If $G \subseteq\Aut(C)$ and $\gr \in G$ has order $m$, then
  \begin{gather*}
    | \Fix_\nu (\gr ) | = | C_G(\gr ) | \cd \sum_{
      \substack{ 1 \leq i \leq      r , \\
        m | m_i ,\\
        \gr \sim_G x_i ^{m_i \nu / m}}} \frac{1}{m_i}.
  \end{gather*}
\end{lemma}
(Here $C_G(\gr)$ denotes the centralizer of $\gr$ in $G$ and $\sim_G$
denotes the equivalence relation given by conjugation in $G$.)  This
lemma follows from \cite[Theorem 7]{Ha71}, see also \cite[Lemma
11.5]{breuer}.

\begin{say}\label{say_rep}
  Given a $G$-Galois cover $C \rightarrow \PP^1$ let $ \rho\colon G
  \longrightarrow {\rm GL}(H^0(C, K_C)) $ denote the representation on
  holomorphic 1-forms and let $\chi_\rho$ be the character of
  $\rho$. Notice that up to equivalence the representation $\rho$ only
  depends on the data $(\mm, G, \theta)$, not on the parameter $t\in
  Y_r$.
\end{say}

\begin{teo}
[Eichler Trace Formula]
Let $\gr $ be an automorphism of order $m>1$ of a Riemann surface $C$
of genus $g>1$. Then
\begin{equation}\label{eq_EichlerFormula}
  \chi_\rho(\gr )=
  \spur(\rho({\gr}))=1+\sum_{P\in \Fix(\gr)}\frac{\zeta_P(\gr)}{1-\zeta_P(\gr)}.
\end{equation}  
\end{teo}
(See e.g.  \cite[Thm. V.2.9, p. 264]{FK}.)  Collecting the terms with
equal exponent and using the previous lemma one gets the following.
\begin{cor}
  \begin{equation}
    \label{carg}
    \chi_\rho(\gr )=
    1+
    |C_G(\gr)| 
    \sum_{\nu \in I(m)} 
    \biggl\{\sum_{\substack{
        1 \leq i \leq      r , \\
        m | m_i ,\\
        \gr \sim_G x_i ^{m_i \nu / m}}}
    \frac{1}{m_i}  \biggl \} \frac{\zeta_m^\nu}{1-\zeta_m^\nu}.
  \end{equation}  
\end{cor}
\begin{say}
  Another corollary of the Eichler Trace Formula is the well known
  Chevalley--Weil formula which gives the multiplicity of a given
  irreducible representation of $G$ in $H^0(X, K_C)$. More precisely,
  denote by $\operatorname{Irr}(G)$ the set of irreducible characters
  of $G$.  For $\chi \in \operatorname{Irr}(G)$ let $\sigma_\chi$ be
  the corresponding irreducible representation and let $d_{\chi}$ be
  the degree of $\sigma_{\chi}$.  Next, denote by $\mu_{\chi}$ the
  multiplicity of $\sigma_\chi$ inside $\rho$.  Moreover, let $x_i$ be
  an element of order $m_i$ in $G$ that represents the local monodromy
  of the covering $C \ra \PP^1$ at the branch point $P_i$ and let
  $E_{i,\alpha}$ denote the number of eigenvalues of
  $\sigma_{\chi}(x_i)$ that are equal to $\zeta^{\alpha}_{m_i}$, where
  $\zeta_{m_i} = e^{2\pi i / m_i} $ as usual.
\end{say}
\begin{teo}[Chevalley--Weil \cite{CW}]
  Let $ C \rightarrow \PP^1$ be a $G$-Galois cover branched at $r$
  points. Let $m_i$ and $E_{i,\alpha}$ be as above. Then the
  multiplicity $\mu_{\chi}$ of a given irreducible character $\chi$ in
  $H^0(C, K_C)$ is
  \begin{equation}\label{eq_ChevWeilFormula}
    \mu_{\chi} = -d_{\chi} +\sum^r_{i=1} \sum^{m_i-1}_{\alpha=0} E_{i,\alpha}\Bigl \langle -\frac{\alpha}{m_i} \Bigr \rangle
    + \eps,
  \end{equation}
  where $\eps=1$ if $\chi$ is the trivial character and $\eps=0$
  otherwise. Here we denote by $\langle q \rangle$ the fractional part
  of $q\in \mathbb{Q}$.
\end{teo}

\begin{say}
  Let $\sigma\colon G \rightarrow {\rm GL}(V)$ be any linear
  representation of $G$ with character $\chi_\sigma$. Denote by
  $S^2\sigma$ the induced representation on $S^2 V$ and by
  $\chi_{S^2\sigma}$ its character.  Then for $x \in G$
  \begin{equation}\label{eq_symchar}
    \chi_{S^2\sigma}  
    (x)=\frac{1}{2}\bigl (\chi_\sigma(x)^2+\chi_\sigma(x^2)\bigr).
  \end{equation}
  (See e.g. \cite[Proposition 3]{S}).
\end{say}

\begin{say}
  We are only interested in the multiplicity $N$ of the trivial
  representation inside $S^2\rho$. We remark that since the
  representation $\rho$ only depends on the datum $(\mm, G, \theta)$,
  the same happens for $N$.  Using the orthogonality relations and
  \eqref{eq_symchar}, $N$ can be computed as follows:
  \begin{gather}
    \label{N}
    N= (\chi_{S^2\rho}, 1) = \frac{1}{|G|} \sum_{x \in G }
    \chi_{S^2\rho}(x) = \frac{1}{2|G|} \sum_{x \in G } \bigl (
    \chi_\rho(x^2) + \chi_\rho(x)^2 \bigr).
  \end{gather}
  Since $\chi_\rho = \sum_{\chi \in \operatorname{Irr}(G)} \mu_\chi
  \chi$ we obtain
  \begin{equation}
    \label{NCW}
    N=\frac{1}{2|G|}\sum_{x \in G}
    \Big(\big(\sum_{\chi \in \operatorname{Irr}(G)}\mu_{\chi}\chi(x)\big)^2+
    \sum_{\chi \in \operatorname{Irr}(G)}\mu_{\chi}\chi(x^2)\Big)
  \end{equation}
  where $ \operatorname{Irr}(G) $ denotes the set of irreducible
  characters of $G$.  Formula \eqref{NCW} is the one used in our
  \verb|MAGMA| script. To computed directly the examples by hand, one
  can use \eqref{N} together with \eqref{carg}.  This is the method
  used in the computation at the end of \S \ref{examples-section}.

\end{say}

\section{Special subvarieties}
\label{Shimura-section}

\begin{say}
  \label{VHS}
  Fix a rank $2g$ lattice $\La$ and an alternating form $E : \La
  \times \La \ra \Zeta$ of type $(1,\lds, 1)$.  For $F$ a field with
  $\QQ \subseteq F \subseteq \C$, set $\La_F : = \La\otimes_\Zeta F$.
  The Siegel upper half-space can be defined as follows
  \cite[Thm. 7.4]{kempf-abelian-theta}:
  \begin{gather*}
    \sieg:= \{J \in \GL (\La_ \R) : J^2 = - I, J^* E = E, E(x,Jx) >0,
    \ \forall x \neq 0 \}.
  \end{gather*}
  The group $\Sp(\La, E)$ acts on $\sieg$ by conjugation and $\ag =
  \Sp(\La, E) \backslash \sieg $.  This space has the structure of
  a smooth algebraic stack and also of a complex analytic orbifold. The
  orbifold structure is the one naturally associated with the properly
  discontinuous action of $\Sp(\La, E) $ on $\sieg$.  Throughout the
  paper we will work with $\ag$ with this orbifold structure.  Denote
  by $A_J$ the quotient $\La_ \R / \La$ provided with the complex
  structure $J$ and the polarization $E$.  On $\sieg$ there is a
  natural variation of rational Hodge structure, with local system
  $\sieg \times \La _ \QQ$ and corresponding to the Hodge
  decomposition of $\La_ \C$ in $\pm i$ eigenspaces for $J$.  This
  descends to a variation of Hodge structure on $\ag$ in the orbifold
  or stack sense.
\end{say}

\begin{say}
  We refer to \S 2.3 in \cite{moonen-oort} for the definition of Hodge
  loci for a variation of Hodge structure.  A \emph{special
    subvariety} $\zg \subseteq\ag$ is by definition a Hodge locus of
  the natural variation of Hodge structure on $\ag$ described above.
  Special subvarieties contain a dense set of CM points and they are
  totally geodesic \cite[\S 3.4(b)]{moonen-oort}. Conversely an
  algebraic totally geodesic subvariety that contains a CM point is a
  special subvariety \cite[Thm. 4.3]{moonen-linearity-1}.  The
  simplest special subvarieties are the \emph{special subvarieties of
    PEL type}, whose definition is as follows (see \cite[\S
  3.9]{moonen-oort} for more details). %, \cite[\S 8]{milne-Shimura}.
  Given $J\in \sieg$, set
  \begin{gather}
    \label{endq}
    \End_\QQ (A_{J}) := \{f\in \End_\QQ(\La_ \QQ): Jf=fJ\}.
  \end{gather}
  Fix a point $J_0 \in \sieg$ and set $D:= \End_\QQ (A_{J_0})$.  The
  \emph{PEL type} special subvariety $\zg (D)$ is defined as
  the image in $\ag$ of the connected component of the set $\{J \in
  \sieg: D \subseteq\End_\QQ(A_J)\}$ that contains $J_0$.
\end{say}

  \begin{lemma}\label{Cartan}
    Let $(M, g)$ be a Riemannian symmetric space of the noncompact
    type. Let $G$ be a group acting isometrically on $(M,g)$. If $M^G$
    is nonempty, then it is a smooth connected submanifold of $M$.
  \end{lemma}
  \begin{proof}
    Fix $x\in M^G$.  Then $G$ acts on $T_xM $ via the differential
    (isotropy action).  The exponential map $\exp_x : T_xM \ra M$ is a
    global diffeomorphism and it is $G$-equivariant with respect to
    the isotropy action on $T_xM$ and the natural action on $M$. Thus
    $M^G = \exp_x( (T_xM)^G)$. Since $(T_xM)^G$ is a linear subspace
    of $T_xM$, $M^G$ is a smooth connected submanifold.
  \end{proof}
  \begin{remark}
    If $G$ is finite, then $M^G$ is always nonempty (Cartan fixed
    point theorem), see \cite[p. 21]{eberlein-libro}.
  \end{remark}
  \begin{cor}
    \label{bert}
    Let $G\subseteq\Sp(\La, E)$ be a finite subgroup. Denote by
    $\sieg^G$ the set of points of $\sieg$ that are fixed by $G$. Then
    $\sieg^G$ is a connected complex submanifold of $\sieg$.
  \end{cor}
  \begin{proof}
    This follows from the Lemma since $\sieg$ is a symmetric space of
    the noncompact type and $\Sp(\La, E)$ acts isometrically and
    holomorphically on $\sieg$.
  \end{proof}
  Another proof of the Corollary follows from \cite[Lemma
  5.2]{van-geemen-1515}.  Set
  \begin{gather}
    D_G:=\{ f\in \End_\QQ (\La_ \QQ) : Jf=fJ, \ \forall J \in
    \sieg^G\}.
    \label{def-DG}
  \end{gather}

\begin{lemma}
  If $J \in \sieg^G$, then $D_G \subseteq\End_\QQ (A_J)$ and the
  equality holds for $J$ in a dense subset of $\sieg^G$.
\end{lemma}
\begin{proof}
  Consider the variation of Hodge structure on $\sieg$ defined in
  \ref{VHS} and restrict it to $\sieg^G$. There is an algebraic
  subgroup $M \subseteq\operatorname{CSp}(2g, \QQ)$ such that the
  Mumford-Tate group $\mt(A_J)$ is contained in $ M$ for any $J\in
  \sieg^G$ and $\mt(A_J) = M$ for $J$ in a dense subset $ \Omega
  \subseteq\sieg^G$.  The complement of $\Omega$ is a countable union
  of analytic subsets.  Recall that
  \begin{gather}
    \label{mumford}
    \End_\QQ(A_J) = \End_\QQ(\La_ \QQ) ^{\mt(A_J)}.
  \end{gather}
  So $\End_\QQ(\La_ \QQ) ^{M} \subseteq\End_\QQ(A_J) $ for any $J\in
  \sieg^G$, with equality for $J \in\Omega$.  It follows immediately
  from \eqref{endq} and \eqref{def-DG} that
  \begin{gather*}
    D_G = \bigcap_{J\in \sieg^G} \End_\QQ(A_J) = \End_\QQ(\La_ \QQ)^M
  \end{gather*}
  and that $D_G = \End_\QQ(A_J) $ for any $J\in \Omega$.
\end{proof}
\begin{prop}
  \label{prop-Bert}
  The image of $\sieg^G$ in $\ag$ coincides with the PEL subvariety
  $\zg (D_G)$.
\end{prop}
\begin{proof}
  Set
  \begin{gather*}
    Y : = \{J \in \sieg: D_G \subseteq\End_\QQ(A_J)\} = \{J\in \sieg:
    fJ=Jf , \forall f\in D_G\}.
  \end{gather*}
  Since $G\subseteq D_G$, we get immediately that $ Y \subseteq
  \sieg^G$. Conversely, if $J\in \sieg^G$, then by definition $D_G
  \subseteq\End_\QQ(A_J)$, i.e. $J \in Y$.  Thus $Y=\sieg^G$.  By the
  previous lemma there is $J_0 \in Y$ such that $D_G
  = \End_{\QQ}(A_{J_0})$. Thus the image of $Y$ in $\ag$ is indeed the
  special subvariety $\zg(D_G)$.
\end{proof}

Recall that $N = \dim \left ( S^2H^0(C,K_C)\right )^G $ and that
$\zg(\mm, G, \theta)$ is defined in \ref{family}.
\begin{lemma}
  \label{dimensione}
  If $J \in \sieg^G $, then $\dim \sieg^G = \dim \zg(D_G) = \dim (S^2
  \La _\R)^G$ where $\La_\R$ is endowed with the complex structure
  $J$.
\end{lemma}
\begin{proof}
  By Lemma \ref{Cartan} it is enough to compute the dimension of $(T_J
  \sieg)^G$.  But $T_J\sieg = S^2\La_\R$, where $\La_\R$ is endowed
  with the complex structure $J$.
\end{proof}

\begin{teo}
  \label{criterio}
  Fix a datum $\datum$ and assume that
  \begin{gather}
    \tag{$\ast$} N = r-3.
  \end{gather}
  Then $\zg (\mm, G, \theta)$ is a special subvariety of PEL type of
  $\ag$ that is contained in $\tor$ and such that $\zg(\mm, G,
  \theta)\cap \jac \neq \vacuo$.
\end{teo}
\begin{proof}
  Let $\mathcal{C} \ra T_g^G$ be the universal family as in \ref{family}.  For any
  $t\in T_g^G$, $G$ acts holomorphically on $C_t$, so it maps
  injectively into $\Sp(\La, E)$, where $\La = H_1(C_t,\Zeta)$ and $E$
  is the intersection form.  Denote by $G'$ the image of $G$ in
  $\Sp(\La, E)$. It does not depend on $t$ since it is purely
  topological.  The Siegel upper half-space $\sieg$ parametrizes
  complex structures on the real torus $\La _ \R / \La = H_1(C_t, \R)
  / H_1(C_t, \Zeta)$ compatible with the polarization $E$.  The period
  map associates to the curve $C_t$ the complex structure $J_t$ on
  $\La _ \R$ obtained from the splitting $ H^1(C_t, \C) = H^{1,0}(C_t)
  \oplus H^{0,1}(C_t) $ and the isomorphism $H_1(C_t, \R)^* _ \C =
  H^1(C_t, \C)$.  Since $G$ acts holomorphically on $C_t$, the complex
  structure $J_t$ is invariant by $G'$. This shows that $J_t \in
  \sieg^{G'}$, so the Jacobian $j(C_t) $ lies in $\zg(D_{G'})$. This
  shows that $\zg (\mm, G, \theta) \subseteq \zg(D_{G'})$. Since
  $\zg(D_{G'})$ is irreducible (e.g. by Corollary \ref{bert}), to
  conclude it is enough to check that they have the same
  dimension. The dimension of $\zg(\mm, G, \theta)$ is $r-3$, see \ref{family}.  By
  Lemma \ref{dimensione}, if $J \in \sieg^{G'}$, then $\dim
  \zg(D_{G'}) = \dim\sieg^{G'} = \dim ( S^2 \La _ \R ) ^{G'}$, where
  $\La _ \R$ is endowed with the complex structure $J$.  If $J$
  corresponds to the Jacobian of a curve $C$ in the family, then $(
  S^2 \La _ \R ) ^{G'}$ is isomorphic to the dual of $(S^2 H^0(C,
  K_C))^G$.  Thus $\dim \zg(D_{G'} )= N$ and \eqref{bona} yields the
  result.
\end{proof}

Although we are mainly interested in positive dimensional families, we
note the following corollaries, which might be of independent
interest.
\begin{cor}
  \label{ACM}
  Let $A$ be a principally polarized abelian variety. Let $G$ be a
  finite group of automorphisms of $A$ that preserve the
  polarization. If
  \begin{gather}
    \label{s2a}
    \left ( S^2 H^0(A, \Omega^1_A) \right )^G = \{0\},
  \end{gather}
  then $A$ has complex multiplication.
\end{cor}
\begin{proof}
  Set $\La := H_1(A, \Zeta)$. Assume that $A$ equals $\La_\R / \La$
  provided with some $J_0\in \sieg$ and that $G \subset\Sp(\La,
  E)$. As in the previous proof, \eqref{s2a} implies that
  $\sieg^G=\{J_0\}$.  The result follows immediately from Proposition
  \ref{prop-Bert}, since special varieties contain CM points.
  Nevertheless a more direct argument can be given as follows.  The
  elements of $\sieg$ correspond to morphisms $h : \mathbb{S} \ra
  \operatorname{CSp}(\La, E)$. If $J$ corresponds to $h$, the
  Mumford-Tate group of $A_J$ is the smallest algebraic subgroup of
  $\GL(\La_\C )$ that is defined over $\QQ$ and contains
  $h(\mathbb{S})$.  Denote by $h_0$ the morphism corresponding to
  $J_0$ and by $M_0$ the Mumford-Tate group of $A_{J_0}$.  We claim
  that the set of morphisms $h $ with $h(\mathbb{S}) \subseteq M_0$
  reduces to $h_0$.  Indeed if $h$ corresponds to $J$ and
  $h(\mathbb{S}) \subseteq M_0$, then $\mt(A_{J}) \subseteq M_0$.  So
  using \eqref{mumford}
  \begin{gather*}
    \End_\QQ (A_{J_0}) = \End_\QQ (\La_\QQ) ^{M_0} \subseteq
    \End_\QQ (\La_\QQ) ^{\mt(A_J)} =
    \End_\QQ (A_{J}).
  \end{gather*}
  Thus $G \subseteq \End_\QQ (A_{J})$, so $J \in \sieg^G$, $J=J_0$ and
  $h=h_0$ as claimed.  If $g\in M_0(\R)$, the morphism $gh_0 g\meno$
  clearly maps $\mathbb{S} $ to $M_0$. By the above $gh_0g\meno =
  h_0$. So $h_0(\mathbb{S}) $ is contained in the center of
  $M_0(\R)$. Since $M_0(\QQ)$ is dense in $M_0(\R)$ \cite [Cor. 18.3
  p. 220]{borel-linear-algebraic}, this center % of $M_0(\R)$
  coincides with the set of real points of the algebraic group
  $Z(M_0)$, which is defined over $\QQ$. Thus $M_0=Z(M_0)$, $M_0$ is a
  torus and $A$ is CM.
\end{proof}

\begin{cor}\label{JCM}
  Let $C$ be a curve and $G $ a subgroup of $\Aut(C)$. If $(S^2 H^0(C,
  $ $K_C))^G$ $ = \{0\}$, then $J(C)$ is an abelian variety of CM
  type.
\end{cor}

\begin{say}
  Using this criterion and the \verb|MAGMA| script, we found some
  examples of CM Jacobians: 10 for $g=2$, 19 for $g=3$, 18 for $g=4$,
  17 for $g=5$, 17 for $g=6$, 23 for $g=7$.  If $C$ and $G$ satisfy
  the hypothesis of Corollary \ref{JCM}, then clearly the
  corresponding family is a point, so $C$ is a curve \emph{with many
    automorphisms}, using the terminology of \cite [Def. 5.17]
  {oort-moduli}. As remarked there it is expected that not every curve
  with many automorphisms be CM. The above criterion identifies a
  subclass of curves with many automorphisms where this is true. It
  would be interesting to check if the 0--dimensional examples with
  $N>0$ (they do exist) are CM or not.
\end{say}

\section{New examples}
\label{examples-section}

\begin{say}
  \label{list}
  In this section we give the list of all new families of Galois
  covers and we explain some of them with more details.  As explained
  in the Appendix the \verb|MAGMA| script computes the list of
    all Hurwitz equivalence classes of data.  Next it decomposes the
    representation on $H^0(C,K_C)$ using the Chevalley-Weil formula
    \eqref {eq_ChevWeilFormula} and computes the number $N$.  The
  complete list of all the families corresponding to the data $(\mm,
  G, \theta)$ with genus $g\leq \numero$ such that $N=r-3>0$ is given
  in Table 2 in the Appendix.
  
  The new examples are the ones in Table 2 with numbers (28)--(40).
  For them we now give a presentation of the Galois group and an
  explicit description of a representative of an epimorphism $\theta$
  (we use the same notation as in \S
\ref {theo: Riemann} and \S \ref {say-ssg}).

  \bigskip
  \begin{center}
    Genus $2$
  \end{center}
  \smallskip

  \begin{itemize}[labelindent=\parindent,leftmargin=*]
  \item [(28)] $S_3 = \sx x, y : y^2=x^3=1, y\meno x y =
    x^2\xs$.   \\
    % $\theta : \Gamma (2,2,3,3) \ra S_3$,\\
    $x_1= y, x_2= y, x_3 = x$, $x_4 = x^2$.
  \item [(29)] $D_4 = \sx x, y : y^2=x^4=1, y\meno x y =
    x^3\xs$.   \\
    % $\theta : \Gamma (2,2,2,4) \ra D_4$,\\
    $x_1= x^3y, x_2= x^2, x_3 = y$, $x_4 = x^3$.
  \item [(30)] $D_6 = \sx x, y : y^2=x^6=1, y\meno x y =
    x^5\xs$.   \\
    % $\theta : \Gamma (2,2,2,3) \ra D_6$,\\
    $x_1= x^3y, x_2= x^4y, x_3 = x^3$, $x_4 = x^4$.
  \end{itemize}

  \smallskip
  \begin{center} {Genus $3$}
  \end{center}
 
  \smallskip

  \begin{itemize}[labelindent=\parindent,leftmargin=*]
  \item [(31)] $S_3 = \sx x, y : y^2=x^3=1, y\meno x y =
    x^2\xs$.   \\
    % $\theta : \Gamma (2,2,2,2,3) \ra S_3$,\\
    $x_1= xy, x_2= x^2y, x_3 = y$, $x_4 = xy$, $x_5 = x^2$.
  \item [(32)] $D_4 = \sx x, y : y^2=x^4=1, y\meno x y =
    x^3\xs$.   \\
    % $\theta : \Gamma (2,2,2,2,2) \ra D_4$,\\
    $x_1= xy, x_2= x^2y, x_3 = x^2$, $x_4 = x^2y$, $x_5 = x^3y$.
  \end{itemize}

  \begin{itemize}[labelindent=\parindent,leftmargin=*]
  \item [(33)] $G=A_4$.
    Set $y_1:=(123)$, $y_2:=(12)(34)$, $y_3:=(13)(24)$.\\
    % $\theta : \Gamma (2,2,3,3) \ra A_4$,\\
    $x_1= y_3=(13)(24), x_2= y_2=(12)(34), x_3 = y_1y_3=(243)$, $x_4 =
    y_1^2y_3=(124)$.
  \item  [(34)] $((\Zeta/4) \times (\Zeta/2)) \rtimes \Zeta/2 = \sx y_1,y_2,y_3 : y_1^2 = y_2^2 = y_3^4 =1, y_2y_3 = y_3 y_2, y_1^{-1} y_2 y_1 = y_2 y_3^2, y_1^{-1} y_3 y_1 = y_3\xs $.   \\
    % $\theta : \Gamma (2,2,2,4) \ra((\Zeta/4) \times (\Zeta/2))
    % \rtimes \Zeta/2$,\\
    $x_1= y_1, x_2= y_1y_2y_3^3, x_3 = y_2y_3^2$, $x_4 =y_3^3$.
  \item [(35)] $G=S_4$.
    Set $y_1:=(12)$, $y_2:=(123)$, $y_3:=(13)(24)$, $y_4:=(14)(23)$.\\
    % $\theta : \Gamma (2,2,2,3) \ra S_4$,\\
    $x_1= y_1y_2^2=(13), x_2= y_3y_4=(12)(34), x_3 = y_1=(12)$, $x_4 =
    y_2^2y_4=(143)$.
  \end{itemize}

  \smallskip
  \begin{center} {Genus $4$}
  \end{center}
  \smallskip

  \begin{itemize}[labelindent=\parindent,leftmargin=*]
  \item [(36)] $Q_8= \sx y_1, y_2, y_3 \ | \ y_1^2 =y_2^2 =y_3 , y_3^2
    = 1,
    y_1 \meno y_2 y_1 = y_2 y_3\xs$.  \\
    % $\theta : \Gamma (2,4,4,4) \ra Q_8$,\\
    $x_1= y_3, x_2= y_2y_3, x_3 = y_1y_2$, $x_4 = y_1y_3$.
  \item [(37)] $G=A_4$.
    Set $y_1:=(123)$, $y_2:=(12)(34)$, $y_3:=(13)(24)$.\\
    % $\theta : \Gamma (2,3,3,3) \ra A_4$,\\
    $x_1= y_3=(13)(24), x_2= y_1=(123), x_3 = y_1=(123)$, $x_4 =
    y_1y_3=(243)$.
  \item [(38)] $(\Zeta/3) \times S_3 = \langle y_1,y_2, y_3 \ | \
    y_1^2 = y_2^3 = y_3^3 = 1, \ y_1 y_2 y_1^{-1} = y_2, y_2 y_3
    y_2^{-1} = y_3, \ y_1 y_3 y_1^{-1} = y_3^2
    \rangle$.   \\
    % $\theta : \Gamma (2,2,3,3) \ra(\Zeta/3) \times S_3$,\\
    $x_1= y_1y_3^2, x_2= y_1y_3, x_3 = y_2y_3$, $x_4 =y_2^2$.
  \end{itemize}

  \smallskip
  \begin{center} {Genus $5$}
  \end{center}
  \smallskip

  \begin{itemize}[labelindent=\parindent,leftmargin=*]
  \item [(39)] $(\Zeta/3) \rtimes \Zeta/4 = \sx y_1, y_3 \ | \
    y_1^4=y_3^3 =1 ,
    y_1 \meno y_3 y_1 = y_3^2\xs$.   \\
    % $\theta : \Gamma (2,3,4,4) \ra(\Zeta/3) \rtimes \Zeta/4$,\\
    $x_1= y_1^2, x_2= y_3, x_3 = y_1^3y_3^2$, $x_4 =y_1^3y_3$.
  \end{itemize}

  \smallskip
  \begin{center} {Genus $7$}
  \end{center}
  \smallskip

  \begin{itemize}[labelindent=\parindent,leftmargin=*]
  \item [(40)] $ y_1 =
    \begin{pmatrix}
      2 & 1 \\ 2 & 0
    \end{pmatrix}
    \ y_2 =
    \begin{pmatrix}
      0&2 \\ 1& 0
    \end{pmatrix}
    \ y_3 =
    \begin{pmatrix}
      1& 2 \\ 2& 2
    \end{pmatrix}
    \ y_4 =
    \begin{pmatrix}
      2 & 0 \\ 0 & 2
    \end{pmatrix}.  $
    \vspace{.2EM}\\
    $\Sl(2,\mathbb{F}_3)= \sx y_1, y_2, y_3, y_4 | y_1^3 =y_4^2 =1 ,
    y_2^2 = y_3^2 = y_4 , y_1 \meno y_2 y_1 = y_3 $, $ y_1\meno
    y_3 y_1 = y_2 y_3 , y_2 \meno y_3 y_2 = y_3 y_ 4\xs$.\\
    % $\theta : \Gamma (2,3,3,3) \ra\Sl(2,\mathbb{F}_3)$,\\
    $% \begin{gather*}
    x_1= y_4, \ x_2= y_1^2 y_2 y_3 y_4 =
    \begin{pmatrix}
      1 &2 \\ 0 & 1
    \end{pmatrix}, \ x_3 = y_1^2 y_2 y_4 =
    \begin{pmatrix}
      1 & 0 \\ 1 & 1
    \end{pmatrix}$,
    \vspace{.2EM}\\
    $ x_4 =y_1^2y_3y_4 =
    \begin{pmatrix}
      2 & 2 \\ 1 & 0
    \end{pmatrix}.  $
  \end{itemize}

  % \newpage

  Now we wish to make some remarks on the geometry of the various
  examples. First of all, we show how to check by hand that the
  examples give indeed special varieties. We do this by explaining in
  detail the computation in two sample examples, namely families (37)
  and (40), see \ref{sample-1} and \ref{sample-2}.  We also show that
  the families (37), (40) and (25) are not contained in the
  hyperelliptic locus (see \ref{HE-1}, \ref{HE-2}, \ref{HE-3}), while
  (8), (22), (36) and (39) are hyperelliptic (see \ref{HE-4}).
  % \sout{Finally we consider inclusions among the families. }
  We show that (25) and (38) have the same image in $\mathsf{M_4}$ and
  in $\mathsf{A}_4$, see \ref{25=38}.
  % \sout{Moreover (38) is contained in (14), which is itself
  % contained in (10), see \ref{inc-1}, \ref{25=38} and \ref{inc-3}. }
  We also note that in the case of family (25) every Jacobian is
  reducible and it is possible to identify explicitely a CM point, see
  \ref{CM-say}.
\end{say}

  The following observation simplifies the computation of $N$.  Denote
  by $G_0$ the set of elements of order 2 in $G$. The set of elements
  of order greater than 2 can be written as $G_1 \sqcup G_1\meno$ for
  some choice of $G_1 \subseteq G$.  Then \eqref{N} becomes
  \begin{multline}
    \label{N2}
    N= \frac {\chi_\rho(1) + \chi_\rho(1)^2} {2|G|}+ \frac{1}{2|G|}
    \sum_{x \in G_0 } ( \chi_\rho(x^2) + \chi_\rho(x)^2) + \\+
    \frac{1}{2|G|} \sum_{x \in G_1 } ( \chi_\rho(x^2) +
    \chi_\rho(x)^2) + \frac{1}{2|G|} \sum_{x \in G_1 } (
    \chi_\rho(x^{-2}) + \chi_\rho(x\meno)^2)
    =\\
    = \frac{ g + g^2 + |G_0| g }{2|G|} + \frac{1}{2|G|} \sum_{x \in
      G_0 } \chi_\rho(x)^2 + \frac{1}{|G|} \sum_{x \in G_1 } \Re \bigl
    ( \chi_\rho(x^2) + \chi_\rho(x)^2 \bigr)
  \end{multline}

\begin{say}
  \label{sample-1}
  Example (37).  One easily checks that $\theta$ is an epimorphism.
  The conjugacy classes of $A_4$ are $\{1\}$, $A:=\{y_1=(123), (134) ,
  (142) , (243)\}$ $B:=\{(132) , (143), (124) , (234)\}$,
  $C=\{y_2=(12)(34), (13)(24), (14)(23)\}$.  We have $G_0= C$ and we
  can set $G_1:=A$ so that $G_1\meno = B$.  It suffices to compute
  $\chi_\rho (y_1) $ and $\chi_\rho(y_2)$. We have $ |C_G(y_1)| =3 $,
  $ |C_G(y_2)| =4$.  Moreover $y_1 \sim_G x_j^{ m_j \nu /3}$ iff
  $\nu=1$ and $j\in \{2,3,4\}$ and $y_2 \sim_G x_j^{ m_j \nu /2}$ iff
  $\nu=1$ and $j=1$.  Using \eqref{carg} one gets $ \chi_\rho (y_1) =
  \zeta_3$, $ \chi_\rho (y_2) = 0$. Hence by \eqref{N2}
  \begin{gather*}
    24 \, N = 4 + 16 + 3 \bigl ( \chi_\rho(y_2)^2 + \chi_\rho(y_2^2) )
    + 8 \Re (\chi_\rho(y_1^2)  + \chi_\rho(y_1)^2 \bigr) =\\
    =32 + 8 \Re (\bar{\zeta_3} + \zeta_3^2) = 24.
  \end{gather*}
  So $N =1$ and by Theorem \ref{criterio} we get a special curve in
  $\mathsf{T}_4$.
\end{say}
\begin{say}
  \label{HE-1}
  We claim that the above family (37) does not contain any
  hyperelliptic curve.  In fact the hyperelliptic involution is
  central in $\Aut(C)$. Hence there is no hyperelliptic involution
  contained in $G$ since its center is trivial.  If there is a
  hyperelliptic involution $\tau$ outside $G$, then $\Fix(\tau)$
  consists of 10 points and it is $G$-invariant, so it is a union of
  $G$-orbits. The only possibility is that it consists of 2 orbits,
  the one over $t_1$, of cardinality 6, and another one of cardinality
  4 over one of the critical values $t_2, t_3, t_4$.  If $p\in
  \pi\meno (t_1)$, then the stabilizer $\Aut(C)_p$ contains $
  G_p\times \sx \tau \xs \cong \Zeta/2 \times \Zeta /2$. This is
  impossible since $\Aut(C)_p$ is cyclic.
\end{say}

\begin{say}
  \label{sample-2}
  Example (40).  One easily checks that $\theta$ is an epimorphism.
  In $\Sl(2, \mathbb{F}_3)$ there are 7 conjugacy classes: $\{1\}$,
  $G_0=\{y_4\} $,
  \begin{align*}
    A &= \{ y_1=
    \begin{pmatrix}
      2 & 1 \\ 2 & 0
    \end{pmatrix}
    , \alpha_1:=\begin{pmatrix} 0 & 1 \\ 2 & 2
    \end{pmatrix}
    , \alpha_2:=\begin{pmatrix} 1 & 1 \\ 0 & 1
    \end{pmatrix}
    , \alpha_3:=
    \begin{pmatrix}
      1 & 0 \\ 2 & 1
    \end{pmatrix}
    \}
    \\
    B &=\{y_1^{-1}, \alfa_1^{-1}, \alpha_2^{-1}, \alpha_3^{-1}\}\\
    C &= \{ y_2, y_3, y_2y_3, y_2\meno, y_3\meno, (y_2 y_3)\meno\}\\
    D &=\{a_1:=
    \begin{pmatrix}
      2 & 1 \\ 0 & 2
    \end{pmatrix}
    , a_2:=
    \begin{pmatrix}
      0 & 1 \\ 2 & 1
    \end{pmatrix}
    , a_3:=
    \begin{pmatrix}
      1 & 1 \\ 2 & 0
    \end{pmatrix}
    , a_4:=
    \begin{pmatrix}
      2 & 0 \\ 2 & 2
    \end{pmatrix}
    \} \\
    F &= \{a_1\meno, a_2\meno, a_3\meno, a_4\meno\}
  \end{align*}
  The elements of $A$ and $B$ have order 3. The elements in $C$ have
  order 4. The elements in $D$ and $F$ have order 6.  Using
  \eqref{carg} one computes
  \begin{gather*}
    \chi_\rho(y_4) = -5 \qquad  \chi_\rho(y_1) = 2 \zeta_3^ 2 -1\\
    \chi_\rho(y_2) =1 \qquad \chi_\rho(a_1) =1.
  \end{gather*}
  Using \eqref{N2} one gets $N=1$, hence this family yields a special
  curve in $\mathsf{T}_7$.
\end{say}
\begin{say}
  \label{HE-2}
  We claim that the above family (40) is not contained in the
  hyperelliptic locus. In fact the center of $G$ is generated by
  $y_4$, which has order 2, but its trace is $-5$. So it does not act
  as $-1$. Assume that $\tau$ is a hyperelliptic involution not
  contained in $G$.  The set of fixed points of $\tau $ has order 16
  and it is $G$-invariant. The orbits of $G$ have cardinality 24, 12
  (only one orbit) or 8 (three orbits). Thus the only possibility is
  that $\tau$ fixes pointwise the fibres over say $t_2$ and $t_3$.  We
  can assume that $t_1 = 1$, $t_2 = 0$ and $t_3 = \infty$. Then $\tau$
  descends to an involution $\hat{\tau}$ of $\PP^1$ fixing both $0$
  and $\infty$ and interchanging $t_1$ and $t_4$. But then necessarily
  $\hat{\tau}(z) = -z$ and $q_4 = -1$. Therefore there is at most one
  hyperelliptic curve in this family.
\end{say}

\begin{say}
  \label{HE-4}
  We now observe that some of the families in Table \ref{data}
  are contained in the hyperelliptic locus. In fact, in example (36) \
  one can check that $\chi_\rho(y_3 ) = -4$. This is enough to
  conclude that $y_3$ is the hyperelliptic involution. Indeed, fix on
  $H^0(C,K_C)$ a $G$-invariant Hermitian product and consider the
  Hermitian product $(A,B) = \spur AB^*$ on
  $\operatorname{End}H^0(C,K_C)$.  Since $\spur \rho(y_3) =
  (\rho(y_3), \id)= -4$, and $\rho(y_3) $ is unitary, the
  Cauchy-Schwarz inequality yields $\rho(y_3) = - \id$. This shows
  that $y_3$ is the hyperelliptic involution and all the family is
  contained in the hyperelliptic locus.

  The same applies to example (39) \ since $\chi_\rho(y_2) = - 5$.  In
  the same way one can check that families (8) and (22) of Table 1 and
  Table 2 in \cite{moonen-oort} are contained in the hyperelliptic
  locus.

\end{say}

\begin{say}
  \label{25=38}
  We claim that families (38) and (25) have the same image in $
  \mathsf{M}_4$ and in $\mathsf{A}_4$. In other words we claim that
  each curve in (25)
%\eqref{NostraZ3Z3} 
admits another
  $\Zeta/2$--action. To do this we first show that (25) is contained
  in family (10) of Table 1 in \cite{moonen-oort}.  This is the family
  of cyclic covers of $\PP^1$ with group $G = \Zeta/3$ and
  ramification data $(3,3,3,3,3,3)$.  An affine equation for this
  cyclic family is the following:
    \begin{equation}
      \label{BoonenZ3}
      y^3 = \prod_{i=1}^6(x-t_i)
    \end{equation}
    where the group action is $(x,y) \mapsto (x, \zeta_3y)$.  Choose
    the critical values as follows: $t_1 = 1, \ t_2 = \zeta_3, \ t_3 =
    \zeta_3^2, \ t_4 = t, \ t_5 = \zeta_3 t, \ t_6 = \zeta_3^2 t$. We
    obtain the one dimensional family
    \begin{equation}
      \label{NostraZ3Z3}
      y^3 = (x^3 -1)(x^3 - t^3)
    \end{equation}
    as $t$ varies in $\C\setminus \{0,1\}$.  This is in fact family
    (25) with group $\Zeta/3 \times \Zeta/3$ and ramification data
    $(3,3,3,3)$, where the action of the second generator is $(x,y)
    \mapsto (\zeta_3 x, \zeta_3y)$.  Now consider the maps $h_i : C_t
    \ra C_t$
    \begin{gather*}
      % \label{acca3}
      h_1(x,y):= ({t}/{x}, {ty}/{x^2})\qquad h_2(x,y) = (x,
      \zeta_3y)\qquad h_3(x,y) = (\zeta_3x, \zeta_3y).
    \end{gather*}
    $h_2$ and $h_3$ give the above action of $\Zeta /3\times
    \Zeta/3$. Together with $h_1$ they yield an action of the group $
    \Zeta/3 \times S_3 = \langle h_1,h_2, h_3 \ | \ h_1^2 = h_2^3 =
    h_3^3 = 1, \ h_1 h_2 h_1^{-1} = h_2, h_2h_3h_2\meno = h_3, \ h_1
    h_3 h_1^{-1} = h_3^2 \rangle$. This action has ramification data
    $(2,2,3,3)$. This proves that (25)=(38).
\end{say}

\begin{say}
  \label{HE-3}
  We note in passing that family (25) does not intersect the
  hyperelliptic locus.  Assume that an element of the family admits a
  hyperelliptic involution $\tau$.  Since $G=\Zeta/3 \times
    \Zeta/3$ does not contain elements of order 2, $\tau \not \in G$.
  Since $\tau$ commutes with $G$, the fixed points of $\tau$ form a
  $G$-invariant set of cardinality $10$.  Since all the orbits of $G$
  have cardinality either 9 or 3, this is impossible.
\end{say}

\begin{say}
  \label{CM-say}
  Note that using the method of \cite[pp. 68-69]{rohde} it is easy to
  identify a CM point in family (25).  Setting $t=-1$ in
  \eqref{NostraZ3Z3}, we get $C_{-1} : = \{y^3=(x^6 -1) =
  \prod_{i=0}^5 (x-\xi_6^i)\}$. Let $V $ be the Fermat curve with
  affine equation $x^{18} - y^{18} -1 = 0$. Then $f(x,y) = (x^3, y^6)$
  is a well-defined non-constant map $f : V \ra C_{-1}$.  By Lemma
  2.4.3 in \cite{rohde} $C_{-1}$ has complex multiplication.
One can also check that the subgroup of order 3 generated by
    $h_3$ %(see \ref{25=38})
    acts freely on $C_t$, with quotient a curve of genus 2. Hence the
    Jacobian of $C_t$ is isogenous to a product for every $t$.
\end{say}

\section{Proof of Theorem \ref{mainC}} \label{id-section}

We already saw in \ref{25=38} that distinct data can give rise to the
same locus in $\mg$ and  in $\ag$.  In this
section we explore this fact systematically and we prove Theorem
\ref{mainC}.

\begin{lemma}
  \label{trick} Assume that two data $(\mm, G, \theta ) $ and $ (\mm',
  $ $ G',$ $ \theta')$ satisfying condition \eqref{bona} give rise to
  the same Shimura variety $\zg(\mm, G, \theta ) = \zg (\mm', G',
  \theta')$.  Then there is a third datum $(\mm'', G'', \theta'')$,
  also satifying \eqref{bona}, such that $\zg(\mm, G,$ $ \theta ) =
  \zg (\mm', G', \theta') =\zg (\mm'', G'', \theta'')$ and such that
  there are monomorphisms $f :G \ra G''$ and $f': G' \ra
  G''$. Moreover if $|f(G) \cap f'(G') | \leq k$, then
  \begin{gather*}
    |G''| \geq \frac{ |G| \cd |G'|}{k}.
  \end{gather*}
\end{lemma}
\begin{proof}
  Let $C$ be a generic curve in the family defined by $\datum$ or
  $(\mm', $ $G',$ $ \theta')$.  Let $G''$ denote the automorphism
  group of $C$.  The quotient $C /G''$ has genus zero and the action
  of $G''$ on $C$ defines a datum $ (\mm'', G'', \theta'')$.  By
  construction there are monomorphisms $f$ and $f'$ as required
  corresponding to the actions of $G$ and $G'$ on $C$. So we can
  consider $G$ and $G'$ as subgroups of $G''$.  The family defined by
  $ (\mm'', G'', \theta'')$ contains the one defined by $\datum$,
  which coincides with the one defined by $(\mm', G',
  \theta')$. Therefore $ r-3 = r'-3 \leq r'' -3$.  Since $N'' : = \dim
  (S^2H^0(C,K_{C}) )^{G''}$ and $G \subseteq G''$, we have
  \begin{gather*}
    r'' -3 \leq N'' \leq N = r -3 \leq r''-3.
  \end{gather*}
  Hence $N=N'=N''=r-3=r'-3=r''-3$. This shows that $\zg(\mm, G, \theta
  ) = \zg (\mm', G', \theta') =\zg (\mm'', G'', \theta'')$.  The last
  statement follows by considering the inclusions of sets $G / G\cap
  G' \hookrightarrow G'' / G'$.
\end{proof}

\begin{teo}
  In genus 2 the data satisfying \eqref{bona} yield the following four
  Shimura subvarieties:
  \begin{center}
    \begin{tabular}[c]{ll}
      $   N =1$ & $(3) =(5) = (28) = (30),$ 
      $(4) = (29).$\\
      $N=2$ &$ (26).$\\
      $N =3 $ &$ (2).$
    \end{tabular}
  \end{center}
  In particular in genus 2 all families are abelian.
\end{teo}
\begin{proof}
  Family (2) coincides with $\mathsf{M}_2$. It is different from all
  other families just by dimension reasons. Similarly (26) is
  different from all other families.

  It remains to deal with the 1-dimensional families.  First we show
  that (3) = (30).  The Galois group of (30) is $D_6$. Using the
  notation of \ref{list} set $H: = \sx x^ 2\xs = \{ 1, x^2 , x^4
  \}$. Since the only elements of order 3 in $D_6$ are $x^2 $ and
  $x^4$, $H$ is a normal subgroup of $D_6$. For a given element $C$ of
  the family (30), set $B:= C /H$.  We claim that $B = \PP^1$.  Denote
  by $\pi : C \ra \PP^1 = C /D_6$ the original covering that defines
  family (30).  The fixed points of $H$ are exactly the points of the
  fibre $\pi\meno(t_4)$. So there are exactly 4 points of $C$ that are
  fixed by $H$. By the Riemann-Hurwitz formula $g(B)=0$. Therefore the
  family (30) is contained in a family of cyclic coverings of the line
  with Galois group $ \Zeta /3$. This family does not necessarily
  satisfy \eqref{bona}. Nevertheless the \verb|MAGMA| script gives the
  list of all families, not only the ones satisfying \eqref{bona}, see
  \ref{say-app}. Looking at this list we conclude that this family
  must be (3), so (3) = (30).

  Since every genus 2 curve is hyperelliptic, the family (3) must be
  contained in some family satisfying \eqref{bona} with Galois group
  $\Zeta /3 \times \Zeta /2 = \Zeta /6$. There is only one such
  family, namely (5). Therefore (3) =(5).  The same reasoning shows
  that (28) must be contained in a family with group $S_3 \times \Zeta
  /2 = D_6$. Again there is only one family with this Galois group, so
  (28) = (30).  We have proven (3) = (5) = (28) = (30).

  The group $D_6$ is maximal in the list of possible groups. If we had
  (4) $= $ (30), then by Lemma \ref{trick} there should exist a
  monomorphism $\Zeta /4 \hookrightarrow D_6$. Since this is not
  possible, (4) $\neq $ (30).

  Finally we check that (4) = (29).  Inside $D_4$ consider the
  subgroup $H:=\sx x \xs \cong \Zeta /4$.  Let $C$ be an element of
  the family (29) with covering map $\pi : C\ra \PP^1 = C/D_4$. Set
  $B:= C / H$. The ramification of the projection $C \ra B$ is given
  by $\pi\meno(t_2) \cup \pi\meno (t_4)$. The first fiber consist of 4
  points with stabilizer %(in $H$)
  $\sx x^2 \xs$. The second fiber consists of 2 points with
  stabilizer % (in $H$)
  $H$. By Riemann-Hurwitz we get that $g(B) =0$.  Thus (29) is
  contained in a family with structure group $\Zeta /4$.  This family
  does not necessarily satisfy \eqref{bona}, nevertheless using the
  list of all families obtained using the \verb|MAGMA| script (see
  \ref{say-app}), we conclude that there is only one such family,
  namely (4). Thus we get (29) = (4).
\end{proof}

\begin{teo}
  \label{id-3} In genus 3 the data satisfying \eqref{bona} yield the
  following 9 distinct Shimura subvarieties:
  \begin{center}
    \begin{tabular}[c]{ll}
      $   N =1$ & 
      $  (7) = (23) = (34)$, $ (9)$,  $  (22),$  $  (33) = (35)$.     \\
      $N=2$ & $  (6)$,        $(8)$,   $(31)$,
      $(32)$. \\
      $N =3 $ &$ (27).$ 
    \end{tabular}
  \end{center}
  In particular there are 3 new non-abelian Shimura families.
\end{teo}
\begin{proof}
  Since (27) is the only family of dimension 3, it is clearly distinct
  from all the others.

  There are 4 families of dimension 2: (6), (8), (31) and (32).  We
  want to prove that they are all different from each other.  Since
  $S_3$ and $D_4$ are maximal within groups in these families, Lemma
  \ref{trick} implies that (31) $\neq $ (32).  Similarly (6) $\neq $
  (8) since there is no group $G''$ appearing in these 4 families with
  $ |G''| \geq 12$.  Moreover $D_4$ does not contain a subgroup
  isomorphic to $\Zeta /3$, so $(6) \neq (32)$. And similarly $(8)
  \neq (31)$.

  We now check that (31) $\neq$ (6). If $G$ acts on a curve $C$ and $H
  \subseteq G$ is a subgroup, then the representation of $G$ on
  $H^0(C,K_C)$ obviously restricts to the reprensentation of $H$ on
  $H^0(C,K_C)$, so $\tr (\rho(H)) \subseteq \tr(\rho(G))$.  The
  computation using the \verb|MAGMA| script gives the full character
  of $\rho$ for both families (see \ref{say-app}) and one can check
  that this does not happen.  Another way of seeing this would be to
  check that the unique subgroup $H \subset S_3$ of order 3, which is
  $H = \sx x \xs$, acts on an element $C$ in the family (31), in such
  a way that $C/H$ has genus 1.

  Finally we check that (32) $\neq$ (8).  One can just observe that
  (8) is hyperelliptic, while from the character of $\rho(D_4)$ it
  follows that $D_4$ does not contain any hyperlliptic
  involution. Since there is no familiy with group containing $D_4
  \times\Zeta /2$, it follows that $(32) $ is not hyperelliptic, hence
  $(8) \neq (32)$.  By the same argument one shows that also family
  (31) is not hyperelliptic.  This completes the analysis of
  2-dimensional families.

  There are 7 data yielding families of dimension 1 and we want to
  prove that they yield exactly four distinct families as follows:
  \begin{gather*}
    (7) = (23) = (34) \qquad (33) = (35) \qquad (9) \qquad (22) .
  \end{gather*}
  First observe that (34) is not hyperelliptic.  By looking at the
  character of $\rho$ one can see that there is no hyperelliptic
  involution contained in $G =((\Zeta/4) \times (\Zeta/2)) \rtimes
  \Zeta/2 $.  Since $G$ is maximal among the groups of the genus 3
  families, there is no family with group $G\times \Zeta /2$. Thus
  (34) is not hyperelliptic.

  Next we show that $(7) = (34)$.  Set $H:= \sx y_3 \xs =Z(G)$. For
  $C$ an element of the family (34), one can check as above that $g(C
  /H) = 0$. So $(34)$ is included in a family with Galois group $\Zeta
  /4$. This family does not necessarily satisfy $N=r-3$. From the
  complete list of data in genus 3, one sees that there are 3 such
  families. One can check that two of them are hyperelliptic.  The
  third one is (7). Since (34) is not hyperelliptic it follows that
  $(34) = (7)$.

  The same argument shows that (23) = (34). In fact take $H:= \sx y_2,
  y_3 \xs \subseteq G$.  One can check that $C/H = \PP^1$. So (34) is
  also contained in a family with group $\Zeta/4 \times \Zeta /2$.
  There are 3 such families and 2 of them are hyperelliptic. So (34)
  must coincide with the third, which is (23).

  The same argument as above shows that $(33) = (35)$.  Indeed, if $C$
  in an element of (35), then $C /A_4 = \PP^1$, so (35) is contained
  in another family (not necessarily with $N=r-3$) with group $A_4$.
  Since (33) is the unique such family, we conclude (35) = (33).

  (34) $\neq $ (35) since both groups are maximal.

  (35) $\neq $ (9) since $S_4$ is maximal and contains no elements of
  order $6$.

  (7) $\neq $ (9) since the only groups with order a multiple of 12
  are $A_4$ and $S_4$, but $(33) = (35) \neq (9)$.

  (34) $\neq $ (22) since (22) is hyperelliptic and (34) is not.

  (35) is not hyperelliptic, since $S_4$ is centerless and maximal. So
  $(35) \neq (22)$.

  Finally $(9) \neq (22)$. Otherwise by Lemma \ref{trick} they would
  be equal to a family with a Galois group $G''$ of order at least
  24. The only possibility would be (9) = (22) = (35), which is false.
\end{proof}

\begin{teo}
  In genus 4 the data satisfying \eqref{bona} yield the following 8
  distinct Shimura subvarieties:
  \begin{center}
    \begin{tabular}[c]{ll}
      $   N =1$ & 
      $  (11)$,  $  (12)$,       $  (13) = (24),$ 
      $  (25) = (38), $ $  (36),$ 
      $ (37) .$ \\
      $N=2$ & $   (14)$. \\
      $N =3 $ &$ (10).$ 
    \end{tabular}
  \end{center}
  In particular there are 2 non-abelian Shimura families.
\end{teo}
   
\begin{proof}
  (14) is the only 2--dimensional family and (10) is the only
  3--dimensional one.

  We analyze the 1--dimensional data.  (11) is the only one with
  Galois group of order divisible by 5. So by Lemma \ref{trick} it is
  different from all the other families.

  We already know that (38)=(25) from \ref{25=38}.

  The group $Q_8$ is maximal among the ones appearing as Galois groups
  in the genus 4 families.  (36) is hyperelliptic by \ref{HE-4}.  So
  it is different from (37) and (25) by \ref{HE-1}, \ref{HE-3}.  By
  Lemma \ref{trick} it is also different from all the abelian ones:
  these have either an element of order 5 or an element of order 3.
  Thus (36) is different from all other families.

  By \ref{trick} (37) and (38) are different, since both groups are
  maximal.  We claim that (37) is different from all the abelian ones.
  If (37) is equal to some abelian family, the abelian group must be
  contained in $A_4$. This never happens.  Thus also (37) is different
  from all other families.

  Now consider (24). The group $G=\Zeta / 2 \times \Zeta /6 $ is
  maximal. An epimorphism for this family is given by
  \begin{gather*}
    x_1 = (1,0) \qquad x_2 = (0,3) \qquad x_3 = (0,2) \qquad x_4 =
    (1,1).
  \end{gather*}
  (Compare with Table 2 in \cite{moonen-oort}.)  The subgroups of $G $
  isomorphic to $\Zeta /6$ are
  \begin{gather*}
    H_1 = \sx (0,1) \xs, \qquad H_2 = \sx (1,1) \xs \qquad H_3 = \sx
    (1,2) \xs.
  \end{gather*}
  One can check that for any element $C$ of the family (24) and for
  any $i=1,2,3$ we have $C /H_i = \PP^1$. Moreover the map $C \ra
  C/H_2$ has ramification data $\mm=(3,3,6,6)$ and has monodromy ${a}
  = (1, 1, 2,2)$ (notation as in \cite{moonen-oort}). Hence we
  conclude that (24)=(13).  We note in passing that the other maps $ C
  \ra C/H_i$ for $i=1,3$ show that $(24) \subset (14)$.  One
  needs to use the fact that in genus 4 there is a unique family --
  not necessarily satisfying \eqref{bona} -- with group $\Zeta/6$ and
  ramification (2,2,3,3,3); this is family (14).

  Since the group $G$ is maximal, and none of the 3 subgroups $H_i$
  yields a family with ramification $(2,6,6,6)$ we conclude that $(12)
  \neq (24)$.

  By maximality (24) $\neq $ (38).

  Finally we show that $(12) \neq (38)$.  There are three subgroups
  $H_i \subseteq \Zeta /3 \times S_3$, $i=1,2,3$ isomorphic to $\Zeta
  / 6$. One can check that $C / H_i = \PP^1$ for any $i$ and for any
  $C$ in the family (38). But for all $i$ the ramification of the map
  $C \ra C / H_i$ is of type (2,2,3,3,3).  Since (38) is maximal, this
  shows that $(12) \neq (38)$ and also that $(38) \subset (14)$.

\end{proof}

\begin{teo}
  In genus 5 there are exactly two distinct 1-dimensional Shimura
  subvarieties, (15) and (39). The second one is non--abelian.  In
  genus 6 there are exactly two distinct 1-dimensional Shimura
  subvarieties, (17) and (18) and a 2-dimensional one (16). They are
  all cyclic.  In genus 7 there are exactly 3 distinct 1-dimensional
  Shimura subvarieties, (19), (20) and (40). The last one is
  non--abelian.
\end{teo}
\begin{proof}
  In genus 5 there are only two data satisfying \eqref{bona}: (15)
  with group $\Zeta/8$ and (39) with group $\Zeta /3 \rtimes \Zeta
  /4$. By Lemma \ref{trick} they are distinct.  In genus 6 there are
  three data satisfying \eqref{bona}: (16) with group $\Zeta /5$, (17)
  with group $\Zeta /7$ and $(18)$ with group $\Zeta /10$.  The first
  one is 2-dimensional the others are 1-dimensional.  By Lemma
  \ref{trick} the subvariety corresponding to (17) is different from
  the one of (18), while (16) is distinct from the others by
  dimension.  In genus 7 there are three data satisfying \eqref{bona}:
  (19) with group $\Zeta /9$, (20) with group $\Zeta/12$ and (40) with
  group $\Sl(2, \mathbb{F}_3) $. Again by Lemma \ref{trick} they are
  all distinct.
\end{proof}

Finally we notice that the results in this section give the proof of
Theorem \ref{mainC}.

\appendix
\section{}

\begin{say}
  This appendix gives the relevant information on the script and
  contains a table of all the data $(\mm, G, \theta)$ with genus
  $g\leq \numero$ and $N=r-3>0$ up to Hurwitz equivalence.
\end{say}

\begin{say}\label{say-app}
  To perform our calculations we wrote a \verb|MAGMA| \cite{MA}
  script, which is available at:

  \medskip

  \verb|users.mat.unimi.it/users/penegini/|

  \verb|publications/PossGruppigFix_v2Hwr.m|.  \medskip

\noindent

The program performs the following calculations. The first two
steps correspond to the  algorithm already described in \cite{BCGP}. 
\begin{enumerate}[leftmargin=*]
\item For a given group order and genus $g \geq 2$ the first routine
  of the program returns all the local monodromies
  $\mathbf{m}:=(m_1,\ldots ,m_r)$ compatible with the Riemann--Hurwitz
  formula
    \begin{equation*}\label{form.RH} 2g(C) - 2 = |G|\left(-2 +
      \sum_{i=1}^r \left(1 - \frac{1}{m_i}\right)\right).
  \end{equation*}
  These are finite. In fact the value of $\sum_{i=1}^r(1-
    1/m_i)$ is fixed and $m_i \geq 2$. Therefore $r$ is bounded. Since
    $ m_i \leq |G|$, there is a finite number of possibilities.  The
  function in the script that performs this calculation is
  \verb|Signature|.
\item After that, the program calculates all groups $G$ of a fixed
  order and all spherical systems of generators (SSG) for $G$ of a
  fixed type $\mathbf{m}$ up to Hurwitz equivalence.  For more
  details on this see e.g. \cite{penegini2013surfaces}. Here we borrow
  some parts of the script given in \cite{BCGP} (function \verb|FindAllComponents|).  One can find the
  tables of all inequivalent pairs $(G$, SSG$)$ at the web page
\begin{verbatim}
http://users.mat.unimi.it/users/penegini/publications.html
\end{verbatim}
\item For each pair $(G$, SSG$)$ it calculates the multiplicity of
  each irreducible representation of $G$ inside $ \rho\colon G
  \longrightarrow {\rm GL}(H^0(C, K_C)) $ using the Chevalley-Weil
  formula \eqref{eq_ChevWeilFormula}. Here we borrowed parts of the
  script given in \cite{Sw}. The function that performs this calculation is \verb|CW.|

\item It calculates the number $N$ using \eqref{NCW} (function \verb|S2rreFormula|).

%   \end{enumerate}
% 
% 
%  
% 
% 
% 
% 
%   \begin{enumerate}
%   \setcounter{enumi}{4}
\item The final out-come, obtained by the main routines \verb|SSGT| and  \verb|AllCases|, is the list of all data $(\mm, G, \theta)$
  with genus $g\leq \numero$ up to Hurwitz equivalence, together with
  the number $N$, plus some additional information, e.g., if the group
\begin{table}[h]
  \caption{All data}
  \begin{center}
    \begin{tabular}{|c|c|c|c|c|c|c|}
      \hline
      $g(C)$ & $\ |G|\ $  & $  G$  & \texttt{Id} %SmallGroup}
      &  $\mm$ & $\dim$  &  \\
      \hline
      $  1 $  & $2$ & $\Zeta/ 2$ &  G(2,1) &  $(2^4)$  & 1   &  (1)          \\ \hline

      $  1 $  & $4$ & $(\Zeta/ 2) \times (\Zeta/2)$ &  G(4,2) &  $(2^4)$  & 1   &  (21)          \\ \hline

      $  2 $  & $2$ & $\Zeta/ 2$ &  G(2,1) &  $(2^6)$    & 3 & (2)    \\ \hline

      $  2 $  & $3$ & $\Zeta/ 3$ &  G(3,1) &  $(3^4)$    & 1 & (3)    \\ \hline

      $  2 $  & $4$ & $\Zeta/ 4$ &  G(4,1) &  $(2^2, 4^2)$    & 1 & (4)    \\ \hline

      $  2 $  & $6$ & $\Zeta/ 6$ &  G(6,2) &  $(2^2 ,3^2)$    & 1 & (5)    \\ \hline

      $  2 $  & $4$ & $(\Zeta/ 2) \times (\Zeta/2)$ &  G(4,2) &  $(2^5)$    & 2 & (26)    \\ \hline

      $  2 $  & $6$ & $S_3$ &  G(6,1)  &      $(2^2,3^2) $ &1   &  (28)    \\ \hline

      $  2 $  & $8$ & $D_4$ &   G(8,3) &   $(2^3, 4)$   & 1  &  (29)  \\ \hline

      $ 2 $ & $ 12$ &  $D_6$ &               G(12,4) & $(2^3,3)$  &1   & (30) \\ \hline

      $  3 $  &   $3$ &    $\Zeta/ 3$      &        G(3,1) & $(3^5)$ & 2  & (6)   \\ \hline

      $  3 $  &   $4$ &    $\Zeta/ 4$      &        G(4,1) & $(4^4)$ & 1  & (7)   \\ \hline

      $  3 $  &   $4$ &    $\Zeta/4$      &        G(4,1) & $(2^3, 4^2)$ &2  & (8)   \\ \hline

      $  3 $  &   $6$ &    $\Zeta/ 6$      &        G(6,2) & $(2 , 3^2, 6)$ & 1  & (9)   \\ \hline

      $  3 $  &   $4$ &    $(\Zeta/ 2) \times (\Zeta/2)$      &        G(4,2) & $(2^6)$ &3  & (27)   \\ \hline

      $  3 $  &   $6$ & $S_3$ & G(6,1)  & $(2^4,3)$   &2   & (31)   \\ \hline

      $  3 $  &   $8$ & $(\Zeta/2) \times (\Zeta/4)$ &                    G(8,2)  & $(2^2,4^2)$  &1 & (22)  \\ \hline

      $  3 $  &  $8$ & $(\Zeta/2) \times (\Zeta/4)$ &                    G(8,2)  & $(2^2,4^2)$   &  1   & (23) \\ \hline

      $  3 $  &   $8$ & $D_4$ &  G(8,3) &                      $(2^5)$  &2  &(32)\\ \hline

      $  3 $  &     $12$ & $A_4$ & G(12,3)&                    $(2^2, 3^2)$  &1  & (33)  \\ \hline

      $  3 $  &  $16$ & $ \left(\Zeta/4 \times \Zeta/2\right) \rtimes   (\Zeta/2) $ & G(16,13) & $(2^3,4)$  &1 & (34)  \\ \hline

      $  3 $  &  $24$ & $S_4$ &                              G(24,12) & $(2^3,3)$   &1  & (35)  \\ \hline

      $  4 $ &  $3$ &  $\Zeta / 3$ &          G(3,1) & $(3^6)$    &3 & (10) \\ \hline

      $  4 $ &  $5$ &  $\Zeta / 5$ &          G(5,1) & $(5^4)$    &1 & (11) \\ \hline

      $  4 $ &  $6$ &  $\Zeta / 6 $ &          G(6,2) & $(2,6^3)$    &1 & (12) \\ \hline

      $  4 $ &  $6$ &  $\Zeta/ 6$ &          G(6,2) & $(3^2,6^2)$    &1 & (13) \\ \hline

      $  4 $ &  $6$ &  $\Zeta / 6$ &          G(6,2) & $(2^2,3^3)$    &2 & (14) \\ \hline

      $  4 $ &  $8$ &  $Q_8$ &          G(8,4) & $(2,4^3)$    &1 & (36) \\ \hline

      $  4 $ &  $9$ &  $(\Zeta/3 )\times ( \Zeta/3)  $ &       G(9,2) & $(3^4)$  &1 & (25) \\ \hline %WW

      $  4 $ &  $12$ &  $ (\Zeta/6 ) \times (\Zeta/2)$ &   G(12,5) & $(2^2,3,6) $ &1 & (24) \\ \hline

      $  4 $ &  $12$ &  $A_{4}$ &    G(12,3) & $(2,3^3) $ &1 & (37)\\ \hline

      $  4 $ &  $18$ & $(\Zeta/3) \times S_3$ & G(18,3) & $(2^2, 3^2)$  &1 & (38)  \\ \hline

      $ 5$  & $8$ & $\Zeta/8 $ &  G(8,1) & $(2,4, 8^2)$    &1   & (15) \\ \hline

      $ 5$  & $12$ & $(\Zeta/3) \rtimes (\Zeta/4)$ &  G(12,1) & $(2,3,4^2)$    &1   & (39) \\ \hline

      $ 6$  & $5$ & $\Zeta/5 $ &  G(5,1) & $(5^5)$    &2   & (16) \\ \hline

      $ 6$  & $7$ & $\Zeta/7 $ &  G(7,1) & $(7^4)$    &1   & (17) \\ \hline

      $ 6$  & $10$ & $\Zeta/10 $ &  G(10,2) & $(2,5^2,10)$    &1   & (18) \\ \hline

      $7$  & $9$  & $\Zeta/9$ &      G(9,1) & $(3,9^3)$   &1  & (19)\\ \hline

      $7$  & $12$  & $\Zeta/12$ &      G(12,2) & $(2,3,12^2)$   &1  & (20)\\ \hline

      $7$  & $24$  & $\Sl(2,\mathbb{F}_3)$ &      G(24,3) & $(2,3^3)$   &1  & (40)\\ \hline
    \end{tabular}
  \end{center}
  \label{data}
\end{table}
  is cyclic or not, the decomposition of the representation on
  $H^0(C,K_C)$, etc.. The program points out those examples for which
  $N=r-3$.  One can find this information for the families with $r\geq
  4$ at the web page
\begin{verbatim}
http://users.mat.unimi.it/users/penegini/publications.html
\end{verbatim}
\end{enumerate}

A similar script for \verb|GAP4| was used in \cite{matteo2011}.

\end{say}

\begin{say}
  We got 40 data $(\mm, G, \theta)$ with genus $g\leq \numero$ and
  $N=r-3>0$.  In Table \ref{data} we list them in order of increasing
  genus.  The numbers in the last column are given to label the
  families following the numeration already assigned in Table 1 and 2
  in \cite{moonen-oort}.  For $\mm$ we use a compact notation, for
  example $(2^2,3^2) = (2,2,3,3)$.  The column $\dim$ lists the
  dimension of $\zgm$.  The column $\texttt{Id} $ lists the
  $\texttt{IdSmallGroup} $ name of the group in the \verb|MAGMA|
  database.

  Examples (1) and (21) are classical.  Examples (2) -- (20) have
  cyclic Galois group and are already listed in
  \cite[p. 136-137]{rohde}, \cite{moonen-special} and
  \cite{moonen-oort}.  Examples (22) -- (27) have already been found
  in \cite{moonen-oort}.  Professor Xin Lu informed us that (36) has
  already been studied from a different point of view in
  \cite[Ex. 7.2]{lu-zuo-Mumford-prep}.  Notice that we get new data
  only for genus $g=2, 3, 4,5, 7$.
\end{say}

\begin{say}
  For the description of the monodromy (or equivalently of an
  epimorphism) in the non-abelian cases see \S
  \ref{examples-section}. For the abelian cases we refer to Tables 1
  and 2 in \cite{moonen-oort}.
\end{say}

% \newpage

%% fine lista
\def\cprime{$'$}

\end{document}